\documentclass{amsart}

\usepackage{amssymb,amscd,amsmath,url}
\usepackage[arrow,matrix]{xy}

\newtheorem{Thm}{Theorem}
\newtheorem{Prop}[Thm]{Proposition}
\newtheorem{Lem}[Thm]{Lemma}
\newtheorem{Cor}[Thm]{Corollary}

\DeclareMathOperator{\Hom}{Hom}
\DeclareMathOperator{\End}{End}
\DeclareMathOperator{\Ext}{Ext}
\DeclareMathOperator{\Ker}{Ker}

\DeclareMathOperator{\modcat}{mod}
\DeclareMathOperator{\add}{add}
\DeclareMathOperator{\rank}{rank}
\DeclareMathOperator{\rk}{rk}
\DeclareMathOperator{\dimv}{\underline\dim}

\title{The cluster complex of an hereditary artin algebra}
\author{Andrew Hubery}
\date{}

\begin{document}

\begin{abstract}
We show that the cluster complex of an arbitrary hereditary artin algebra has the structure of an abstract simplicial polytope. In particular, the cluster-tilting objects form one equivalence class under mutation.
\end{abstract}

\maketitle

Let $\Lambda$ be a (basic) hereditary artin algebra of rank $n$ over a commutative artin ring $R$. Since the centre of $\Lambda$ is semisimple (so a product of fields), we may also assume that $R$ is semisimple. We denote by $\modcat\Lambda$ the category of finite length right $\Lambda$-modules, and recall that the Krull-Remak-Schmidt Theorem holds in $\modcat\Lambda$. Given $T\in\modcat\Lambda$ we denote by $\add(T)$ the additive subcategory generated by the indecomposable summands of $T$.

A module $T\in\modcat\Lambda$ is called rigid provided that $\Ext^1(T,T)=0$, and is called exceptional if it is both rigid and indecomposable. The module $T$ is called basic if it is a direct sum of pairwise non-isomorphic indecomposable modules. Finally, we call $T$ a tilting module provided that it is rigid and there exists an $\add(T)$-coresolution of $\Lambda$ of length one; that is, there exists a short-exact sequence
\[ 0\to\Lambda\to T_1\to T_2\to0, \quad T_i\in\add(T). \]
We define $\rank T$ for a rigid module $T$ to be the number of pairwise non-isomorphic indecomposable summands of $T$.

Let $T$ be a rigid module. Then $T$ can be completed to a tilting module $T\oplus X$, and $T$ is tilting if and only if $\rank T=n$, the rank of $\Lambda$ \cite{Bongartz,HR}. If $\rank T=n-1$, so that $T$ is an `almost-complete' rigid module, then there is a unique exceptional module $X$ such that $T\oplus X$ is tilting if $T$ is insincere, and there are precisely two such complements if $T$ is sincere \cite{HU1,RS1}.

Using these results, one can construct a pure simplicial complex of rank $n-1$, called the tilting complex, by taking as vertices the isomorphism classes of exceptional modules and as $r$-simplices the basic rigid modules of rank $r+1$ \cite{RS2}. In particular, the facets correspond to the basic tilting modules.

If $Q$ is Dynkin, then the tilting complex is a ball of dimension $n-1$ \cite{RS2}, and in general, the boundary is given by those $(n-2)$-simplices corresponding to insincere rigid modules \cite{HU1}. Moreover, Unger showed in \cite{Unger} that each flag-connected component has a boundary. (Using the language of polytopes \cite{MS} we say that two facets are adjacent if they intersect in a ridge, and this induces a decomposition of the tilting complex into so-called flag-connected components.)

It follows that the tilting complex has a natural completion, with simplices given by pairs $(T,\sigma)$ such that $T$ is basic rigid and $\sigma$ is a set of vertices of $Q$ not contained in the support of $T$. The pair $(T,\sigma)$ is an $r$-simplex, where $r=\rank T+|\sigma|-1$. It is easily checked that this is again a pure simplicial complex of rank $n-1$, and the facets are in bijection with the basic `support-tilting' modules --- those basic rigid modules which are tilting modules when restricted to their support.

Unfortunately this completion of the tilting complex was not studied further. Also, Unger's result does not easily generalise to arbitrary hereditary artin algebras, since it uses in an essential way that all exceptional modules have endomorphism ring the base field.

In the seminal paper of Buan, Marsh, Reineke, Reiten and Todorov \cite{BMRRT}, these ideas resurfaced, but in a different guise. They introduced and studied the cluster category associated to the path algebra of a quiver, defined to be the orbit category for the endofunctor $\tau^-[1]$ on the bounded derived category. In this setting, they showed that the cluster-tilting (or maximal rigid) objects provide a realisation for the combinatorics of the recently discovered cluster algebras of Fomin and Zelevinsky \cite{FZ}. More precisely, the indecomposable direct summands of the cluster-tilting objects are naturally in bijection with the cluster variables in such a way that the basic cluster-tilting objects correspond to the clusters and the mutation of clusters corresponds to the two possible completions of an almost-complete rigid object.

As before, one can form the cluster complex, the simplicial complex whose $r$-simplices correspond to basic cluster-tilting objects of rank $r+1$. This is immediately seen to be isomorphic to the above completion of the tilting complex for the quiver. In particular, the pair $(T,\sigma)$ corresponds to the cluster-tilting object $T\oplus P_\sigma[1]$, where $P_\sigma$ is the projective cover of the semisimple module having support $\sigma$. An essential theorem in \cite{BMRRT} is that this cluster complex is flag connected (so that the cluster-tilting objects form a single equivalence class under mutation). At the heart of this lies Unger's aforementioned result, although the proof in \cite{BMRRT} is rather indirect and uses a second result by Happel and Unger \cite{HU2}. In fact, Unger's result could have been applied directly, using that the cluster complex is naturally the completion of the tilting complex.

Of course, one can also define cluster algebras and cluster categories for (skew-) symmetrisable generalised Cartan matrices, which in turn correspond to species (or more general hereditary artin algebras). One would therefore like to extend this correspondence to all hereditary artin algebras, showing that there is always a bijection between the isomorphism classes of support-tilting modules (equivalently, cluster-tilting objects) and the clusters in the corresponding cluster algebra.

In a different direction, Chapoton, Fomin and Zelevinsky \cite{CFZ} gave explicit realisations of the cluster complexes arising from cluster algebras of finite type, thus proving that they are all simplicial polytopes. It is therefore an interesting question as to whether all cluster complexes coming from hereditary artin algebras share this extra structure.

In this article we prove in Theorem \ref{Main Thm} that, for an arbitrary hereditary artin algebra, the cluster complex (the completion of the tilting complex) has the structure of an abstract simplicial polytope. In particular, we prove that this complex is always strongly flag connected, and hence that the set of support-tilting modules forms a single equivalence class under mutation. Again, under the natural bijection between support-tilting modules and cluster-tilting objects, this proves the equivalent statement for cluster categories.

Our proof is very different to \cite{BMRRT}, and does not rely on Unger's theorem \cite{Unger}. In fact our approach has a lot in common with \cite{CB2,Ringel2}, where it is shown that the braid group acts transitively on the set of exceptional sequences. In particular, we show that the co-face of a pair $(T,\sigma)$ (that is, the poset of pairs containing $(T,\sigma)$) is isomorphic to the simplicial polytope for the perpendicular category. More precisely, let $\Lambda_\sigma$ be the factor algebra of $\Lambda$ given by the idempotent $e_\sigma=\sum_{i\in\sigma}e_i$. We define the perpendicular category $(T,\sigma)^\perp$ to be the relative perpendicular category $T^\perp$ inside $\modcat\Lambda_\sigma$. This perpendicular category is known to be equivalent to $\modcat\Gamma$ for some basic hereditary artin algebra $\Gamma$ with $\rank\Gamma=\rank\Lambda-\rank T-|\sigma|$.

This result can be seen as a generalisation of two important results. Firstly, if $T$ is a basic rigid module with $n-1$ indecomposable summands, then $T^\perp\cong\modcat\Gamma$, where $\Gamma$ is a hereditary algebra of rank one, so a division ring. Thus there is a unique basic tilting module and a unique vertex for $\Gamma$, whence its cluster complex is a 1-simplex, or line segment. This can be interpreted as saying that each almost-complete rigid module has exactly two completions to a support-tilting module. Similarly, if $T$ has $n-2$ summands, then this polytope is isomorphic to the link of $T$, and is given by the cluster complex of a rank two hereditary algebra. If we are working over an algebraically closed field, then we have a quiver of type $A_1\cup A_1$ or $A_2$, or else a generalised Kronecker quiver, giving respectively a square, a pentagon or a doubly infinite line. These are precisely the complexes found by Unger \cite{Unger} (taking into account that our complex is connected since we also include support-tilting modules). More generally, for an arbitrary hereditary artin algebra, we also obtain hexagons and octagons, in types $B_2$ and $G_2$ respectively.

Now, given a support-tilting module $T$ and a direct summand $T'$, we see that $T$ is mutation equivalent to $T'\oplus B$, where $B$ is the relative Bongartz complement of $T'$, constructed by first restricting to the support of $T'$. (Moreover, we need only mutate at those indecomposable summands not involved in $T'$.) The point is therefore to construct some measure on support-tilting modules which, by this construction, necessarily decreases. We are then done by induction, since every support-tilting module will be mutation equivalent to the zero module (i.e. the pair $(0,\sigma)$ where $\sigma$ is the set of all vertices).

The measure we use is easy to construct when working over an algebraically closed field --- we take
\[ \lambda\big(\bigoplus_{i=1}^rT(i),\sigma\big) := (0,\ldots,0,\dim T(1),\ldots,\dim T(r)), \]
where there are $|\sigma|=n-r$ zeros and $\dim T(i)\leq\dim T(i+1)$, and order these $n$-tuples lexicographically. In general, taking the length over the centre of $\Lambda$ does not work. For example, let $K/k$ be a field extension of degree 3 and consider the $k$-algebra
\[ \Lambda=\begin{pmatrix}K&K\\0&k\end{pmatrix} \]
of type $G_2$. Then the Auslander-Reiten quiver of $\modcat\Lambda$ is given by
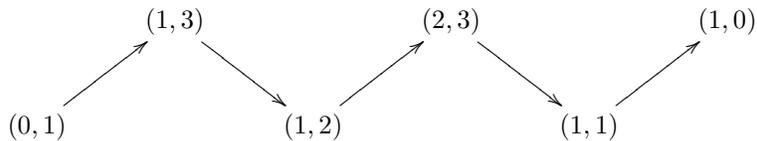
\[ \xymatrix{&(1,3)\ar[dr]&&(2,3)\ar[dr]&&(1,0)\\(0,1)\ar[ur]&&(1,2)\ar[ur]&&(1,1)\ar[ur]} \]
If we consider just the $k$-dimensions of the modules, then we obtain
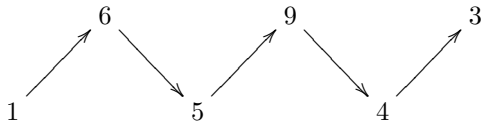
\[ \xymatrix{&6\ar[dr]&&9\ar[dr]&&3\\1\ar[ur]&&5\ar[ur]&&4\ar[ur]} \]
We see that for the indecomposable of dimension vector $(1,2)$, both its Bongartz complement $(1,3)$ and its dual Bongartz complement $(2,3)$ have larger dimension.

To fix this problem, we show in Theorem \ref{total-order} and Corollary \ref{rank2} that taking $\dim_kX/\sqrt{\dim_k\End(X)}$ works. For the algebra $\Lambda$ above, the square of this gives
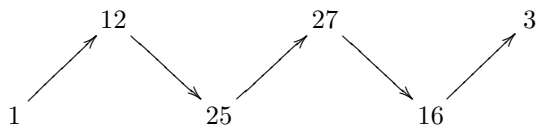
\[ \xymatrix{&12\ar[dr]&&27\ar[dr]&&3\\1\ar[ur]&&25\ar[ur]&&16\ar[ur]} \]

For completeness, we have included in the first few sections a survey of the results we shall need, as well as complete proofs. Our reasons for doing so are that many of the results in the literature are either only stated for for path algebras of quivers over algebraically closed fields, or have only been proved in special cases.

In the first section we unify the ideas of left $\add(T)$-approximations \cite{AS} and universal $\add(T)$-extensions \cite{Bongartz} to all extension groups.

Proposition \ref{HR-generalisation} is a generalisation of a key lemma in \cite{HR}, and can also be found in \cite{RS1}. In this more general form it implies that, for $M$ indecomposable and $\Ext^1(T,M)=0$, the minimal left $\add(T)$-approximation of $M$ is either a monomorphism or an epimorphism, Corollary \ref{min-left-mono-or-epi}.

Proposition \ref{dim-vector} can be found in \cite{Kerner}, though we give a different proof based only on Proposition \ref{HR-generalisation}. One can use this to give a nice proof of Happel and Ringel's result \cite{HR} that the dimension vectors of the indecomposable summands of a basic rigid module are linearly independent in the Grothendieck group, Corollary \ref{lin-indept}.

Our discussion on complements of tilting modules is more in the spirit of \cite{RS1} than \cite{HU1}, making explicit use of the Bongartz complement and its dual, together with Proposition \ref{HR-generalisation}. However, the importance of the relative Bongartz complement (when one restricts to the support of a rigid module) and the fact that it is a summand of the full Bongartz complement, Proposition \ref{B_1,B_2 Prop}, seems to be new.

Theorem \ref{perp-thm} is a natural extension of a theorem in \cite{Ringel1}. Ringel proved this result for exceptional modules, and we generalise this to all rigid modules. This is also the main ingredient in proving Theorem \ref{co-face}. A version of this result for Calabi-Yau categories is given by Theorem 4.9 in \cite{IY}.

Theorem \ref{proj-gen-inj-cogen} can be found in \cite{CB1,Kerner}. Surprisingly, Theorem \ref{derived-equiv} seems to be new, although the ideas used are all standard.

As mentioned above, our main theorem, Theorem \ref{Main Thm}, was known only for path algebras over algebraically closed fields \cite{BMRRT}. One of the key observations, Theorem \ref{total-order}, is new and would appear to have independent interest. Finally Corollary \ref{Cor-Endos} gives an alternative proof in the special case of tilting modules of the well-known result that the product of the endomorphism rings of the indecomposables in a complete exceptional sequence is isomorphic to $\Lambda/\mathrm{rad}\Lambda$ \cite{CB2,Ringel2}.

The author would like to thank W.~Crawley-Boevey and R.~Marsh for interesting discussions.

\section{Universal extensions}

Let $\Lambda$ be a ring and let $T\in\modcat\Lambda$.

An element $\varepsilon\in\Ext^i(M,T')$ with $T'\in\add(T)$ is called a universal $\add(T)$-extension of $M$ provided that $\varepsilon^\ast\colon\Hom(T',T)\to\Ext^i(M,T)$ is surjective. If $\varepsilon_1,\ldots,\varepsilon_n$ are generators for $\Ext^i(M,T)$ as a left $\End(T)$-module, then $(\varepsilon_j)\in\Ext^i(M,T^n)$ is universal. We call $\varepsilon$ minimal if, whenever $\theta\in\End(T')$ satisfies $\theta\varepsilon=\varepsilon$, then $\theta$ is an automorphism. Recall that $\theta\varepsilon$ is the push-out of $\varepsilon$ along $\theta$.

The case of $i=0$ was introduced by Auslander and Smal\o\ under the name of (minimal) left $\add(T)$-approximations of $M$ \cite{AS,ARS}. The case $i=1$ appears in the work of Bongartz on tilting modules \cite{Bongartz}.

\begin{Prop}\label{universal-extension}
Minimal universal $\add(T)$-extensions exist and are unique up to isomorphism. In fact, if $\varepsilon\in\Ext^i(M,T')$ is a universal $\add(T)$-extension, then there exists a decomposition $T'=I\oplus J$ such that $\varepsilon=\bar\varepsilon\oplus 0\in\Ext^i(M,I)\oplus\Ext^i(0,J)$ and with $\bar\varepsilon$ minimal.
\end{Prop}

\begin{proof}
Let $\varepsilon\in\Ext^i(M,T')$ be a universal $\add(T)$-extension of $M$ and suppose $\theta\in\End(T')$ satisfies $\theta\varepsilon=\varepsilon$. Take such a $\theta$ with $I:=\theta(T')$ having minimal length and write $\theta\colon T'\overset{\pi}\twoheadrightarrow I\overset{\iota}\hookrightarrow T'$. Set $\bar\varepsilon:=\pi\varepsilon$, so $\bar\varepsilon^\ast\colon\Hom(I,T)\to\Ext^i(M,T)$ is onto. Let $\bar\phi\in\End(I)$ satisfy $\bar\phi\bar\varepsilon=\bar\varepsilon$. Then $\phi:=\iota\bar\phi\pi\in\End(T')$ satisfies $\phi\varepsilon=\varepsilon$ and $\phi$ has image contained in $I$. By the choice of $\theta$ we see that $\bar\phi$ is an automorphism of $I$. Applying this to $\bar\phi=\pi\iota$ we deduce that $\pi$ is a split epimorphism, so $I\in\add(T)$ and $\bar\varepsilon=\pi\varepsilon\in\Ext^i(M,I)$ is a minimal universal $\add(T)$-extension.

Let $\varepsilon\in\Ext^i(M,T')$ and $\eta\in\Ext^i(M,T'')$ be universal $\add(T)$-extensions of $M$. Then we can write $\eta=\phi\varepsilon$ and $\varepsilon=\phi\eta$ for some $\phi$ and $\psi$. Thus $\varepsilon=\psi\phi\varepsilon$, so if $\varepsilon$ is minimal, then $\psi\phi$ is an automorphism. It follows that minimal universal $\add(T)$-extensions are unique up to isomorphism.
\end{proof}

\begin{Cor}
\begin{enumerate}
\item A left $\add(T)$-approximation $f\colon M\to T'$ is minimal if and only if the cokernel $g\colon T'\twoheadrightarrow C$ is a radical morphism.
\item A universal $\add(T)$-extension $\varepsilon\colon 0\to T'\xrightarrow{f} E\xrightarrow{g}M\to 0$ is minimal if and only if $f$ is a radical morphism.
\end{enumerate}
\end{Cor}

\begin{Cor}\label{ext-corollary}
If
\[ \varepsilon\colon 0\to T'\to E\to M\to0 \quad\textrm{and}\quad \eta\colon 0\to T''\to F\to M\to0 \]
are universal $\add(T)$-extensions of $M$, then $E\cong F$ modulo $\add(T)$.
\end{Cor}

We also have the dual notions of (minimal) universal $\add(T)$-coextensions.

Let $T$ be a fixed module. Given a module we shall use the following notation.
\[ \lambda_M\colon M\to {}_MT \quad\textrm{and}\quad \rho_M\colon T_M\to M \]
will denote the minimal left and right $\add(T)$-approximations of $M$, and
\[ \varepsilon_M\colon 0\to {}_MT'\xrightarrow{a_M} {}_ME\xrightarrow{b_M} M\to 0 \quad\textrm{and}\quad \eta_M\colon 0\to M\xrightarrow{c_M}E_M\xrightarrow{d_M} T'_M\to 0 \]
will denote the minimal universal $\add(T)$-extension and coextension of $M$.

\section{Hereditary Artin Algebras}

From now on, $\Lambda$ will be a basic hereditary artin $R$-algebra, where $R$ is a commutative semisimple artin ring. We have the duality $D\colon\modcat\Lambda\to\modcat\Lambda^{\mathrm{op}}$, $M\mapsto\Hom_R(M,R)$. For basic definitions and results, we refer the reader to \cite{ARS}.

Write $1=e_1+\cdots+e_n$ as a sum of primitive orthogonal idempotents. We then have a complete set of indecomposable projective modules $P(i)=\Lambda e_i$, injective modules $I(i)=D(e_i\Lambda)$, and simple modules $S(i)=P(i)/\mathrm{rad}P(i)=\mathrm{soc}I(i)$. We observe that
\[ \Lambda/\mathrm{rad}\Lambda \cong \prod_i\End(P(i)) \cong \prod_i\End(S(i)) \cong \prod_i\End(I(i)) \]
as products of skew-fields. We write $\rank\Lambda:=n$.

The Grothendieck group $K_0(\Lambda)\cong\mathbb Z^n$ of $\Lambda$ has basis $\epsilon_i:=\dimv S(i)$. Note that $\dimv M=\sum_im_i\epsilon_i$, where $m_i$ is the Jordan-H\"older multiplicicty of $S(i)$ in $M$. Alternatively, $m_i=\dim_{\End(P(i))}\Hom(P(i),M)$, the dimension as a right $\End(P(i))$-module. The Euler form on $\modcat\Lambda$ is given by
\[ \langle M,N\rangle = \ell_R\Hom(M,N)-\ell_R\Ext^1(M,N), \]
where $\ell_R$ denotes the length as an $R$-module. Since $\Lambda$ is hereditary, this descends to a bilinear form on the Grothendieck group $K_0(\Lambda)$.

Let $\tau=D\Ext^1(-,\Lambda)$ be the Auslander-Reiten translate in $\modcat\Lambda$, so
\[ \Hom(N,\tau M) \cong D\Ext^1(M,N) \cong \Hom(\tau^-N,M). \]

Given $\sigma\subset\{1,\ldots,n\}$ we define $e_\sigma:=\sum_{i\in\sigma}e_i$ and $\Lambda_\sigma:=\Lambda/(e_\sigma)$. We observe that $\modcat\Lambda_\sigma\subset\modcat\Lambda$ is a full exact subcategory. The objects are those modules $M$ such that $Me_\sigma=0$. Thus this subcategory is closed under extensions, so that $\Lambda_\sigma$ is again a basic hereditary artin $R$-algebra.

For a $\Lambda$-module $T$ write
\[ \sigma_T:=\{i:Te_i=0\}=\{i:\Hom(P(i),T)=0\}, \]
the complement of $\mathrm{supp}\,T$, and set $\Lambda_T:=\Lambda_{\sigma_T}$. Thus $T$ is naturally a sincere $\Lambda_T$-module. We also define the modules
\[ P_T:=\bigoplus_{i\in\sigma_T}P(i) \quad\textrm{and}\quad I_T:=\bigoplus_{i\in\sigma_T}I(i). \]

\begin{Prop}[Riedtmann-Schofield]\label{HR-generalisation}
Let $T$ and $M$ be modules with $M$ indecomposable. Assume further that $\Ext^1(T,M)=0$. Let $f\colon M\to T'$ be a morphism with $T'\in\add(T)$, and write $f\colon M\overset{\pi}\twoheadrightarrow I\overset{\iota}\hookrightarrow T'$. If $\pi$ is a proper epimorphism, then $\iota$ is a split monomorphism.
\end{Prop}

\begin{proof}
Consider the two short exact sequences
\[ \varepsilon\colon 0\to K\to M\xrightarrow{\pi} I\to 0 \quad\textrm{and}\quad \eta\colon 0\to I \xrightarrow{\iota} T'\to C\to 0. \]

Since $\Lambda$ is hereditary, we know that $\pi_\ast\colon\Ext^1(C,M)\to\Ext^1(C,I)$ is onto, so that $\eta=\pi\xi$ is the push-out along $\pi$ of some $\xi\in\Ext^1(C,M)$, say
\[ \xi\colon 0\to M\to E\to C\to 0. \]
This yields a short exact sequence
\[ 0\longrightarrow M\overset{\binom{\pi}{-a}}{\longrightarrow} I\oplus E\overset{(\iota\,b)}{\longrightarrow} T'\longrightarrow 0, \]
which is split since $\Ext^1(T,M)=0$.

If $\pi$ is a proper epimorphism, then, since $M$ is indecomposable, $\iota$ must be a split monomorphism.
\end{proof}

\begin{Cor}[Happel-Ringel]\label{mono-or-epi}
Let $T$ and $M$ be indecomposable modules such that $\Ext^1(T,M)=0$. Then any non-zero $f\colon M\to T$ is either a monomorphism or an epimorphism.

In particular, if $T$ is exceptional, then $\End(T)$ is a skew-field. More generally, if $T$ is basic and rigid, then the quiver of $\End(T)$ contains no oriented cycles.
\end{Cor}

\begin{Cor}\label{min-left-mono-or-epi}
Let $T$ and $M$ be modules with $M$ indecomposable and $\Ext^1(T,M)=0$. Then the minimal left $\add(T)$-approximation $\lambda_M\colon M\to {}_MT$ is either a mono\-morphism or an epimorphism.
\end{Cor}

There are obvious dual results concerning minimal right $\add(T)$-approximations.

\begin{Prop}[Kerner]\label{dim-vector}
Let $T$ and $U$ be rigid such that $\dimv T=\dimv U$. Then $T\cong U$.
\end{Prop}

\begin{proof}
Let $U'$ be an indecomposable summand of $U$. Since $\langle T,U'\rangle=\langle U,U'\rangle>0$, there exists a non-zero map $T'\to U'$ for some indecomposable summand $T'$ of $T$. Dually, given $T'$, there exists a non-zero map $T'\to U'$ for some $U'$.

Let $T'\to U'$ be non-zero. Factoring this via its image $I$, we observe that $\Ext^1(T,I)=0=\Ext^1(I,U)$. Now, since $\langle U,I\rangle=\langle T,I\rangle>0$, there exists a non-zero homomorphism $U''\to I$ with $U''$ an indecomposable summand of $U$. The composition $U''\to I\hookrightarrow U'$ is non-zero, so either $I\cong U'$ or else $U''\hookrightarrow I$ is injective. Similarly, either $T'\cong I$ or else there exists an epimorphism $I\twoheadrightarrow T''$ for some indecomposable summand $T''$ of $T$.

Thus, given any epimorphism $T'\twoheadrightarrow U'$ there exists an epimorphism $U'\twoheadrightarrow T''$.

Also, since the quiver of $\End(U)$ has no oriented cycles, there exists some $U'$ corresponding to a source. Hence if $T'\to U'$ is non-zero, then it is an epimorphism. For, we cannot have a proper monomorphism $U''\hookrightarrow U'$, so $I\cong U'$ in the notation above.

This argument, together with the analogous result given by exchanging $T$ and $U$, shows that there exists an infinite chain of epimorphisms $T_1\twoheadrightarrow U_2\twoheadrightarrow T_3\cdots$. Since we cannot have an infinite chain of proper epimorphisms, we must have that $T'\cong U'$ for some indecomposable summands $T'$ and $U'$ of $T$ and $U$ respectively. Induction on the number of indecomposable summands finishes the proof.
\end{proof}

\begin{Cor}[Happel-Ringel]\label{lin-indept}
Let $T=\bigoplus_{i=1}^rT(i)$ be a basic rigid module. Then the $\dimv T(i)$ are linearly independent in $K_0(\Lambda)$, so $r\leq n$. We set $\rank T:=r$.
\end{Cor}

\begin{proof}
Any linear relation amongst the $\dimv T(i)$ yields an equality $\sum m_i\dimv T(i)=\sum n_i\dimv T(i)$ with the $m_i$ and $n_i$ non-negative integers. Setting $X=\bigoplus_iT(i)^{m_i}$ and $Y=\bigoplus_iT(i)^{n_i}$, Proposition \ref{dim-vector} implies $X\cong Y$. Thus $m_i=n_i$ for all $i$ by the Krull-Remak-Schmidt Theorem.
\end{proof}

\section{Tilting Modules and the Bongartz Complement}

In this section, $T$ will denote a rigid module.

A tilting module $T$ is a rigid module such that there exists a short exact sequence
\[ 0 \to \Lambda \to T_1 \to T_2 \to 0 \]
with $T_i\in\add(T)$. In other words, $\Lambda$ has an $\add(T)$-coresolution of length $1$.

In \cite{Bongartz} Bongartz proved that every rigid module can be completed to a tilting module. For, consider the universal $\add(T)$-coextension of $\Lambda$
\[ \eta_\Lambda\colon 0\to\Lambda\to E_\Lambda\to T'_\Lambda\to 0. \]
By the universal property, $\Ext^1(T,E_\Lambda)\cong\Ext^1(T,T'_\Lambda)=0$. On the other hand, since $\Lambda$ is projective, given any $M$, $\Ext^1(T'_\Lambda,M)=0$ implies $\Ext^1(E_\Lambda,M)=0$. In particular, setting $M=T\oplus E_\Lambda$, we see that $T\oplus E_\Lambda$ is rigid. Since $\Lambda$ has an $\add(T\oplus E_\Lambda)$-coresolution of length $1$, we deduce that $T\oplus E_\Lambda$ is tilting. Let $B\in\add(E_\Lambda)$ be basic and coprime to $T$ (i.e. having no indecomposable summands in common with $T$) such that $T\oplus B$ is tilting module. We call $T\oplus B$ the Bongartz completion of $T$, and $B$ the Bongartz complement.

\begin{Prop}[Happel-Ringel]\label{r=n}
Let $T=\bigoplus_{i=1}^rT(i)$ be a basic rigid module. Then $T$ tilting if and only if $r=n$. In this case, the $\dimv T(i)$ form a basis of $K_0(\Lambda)$.
\end{Prop}

\begin{proof}
Suppose $r=n$ and consider the Bongartz completion $T\oplus B$. Then $T\oplus B$ is rigid, so $B=0$ by Corollary \ref{lin-indept} and $T$ is tilting.

Conversely, let $T$ be tilting. Since $\Lambda$ has an $\add(T)$-coresolution, the same is true of each indecomposable projective $P(j)$. Thus each $\dimv P(j)$ is a linear combination of the $\dimv T(i)$, and since the $\dimv P(j)$ form a basis of $K_0(\Lambda)$, so too do the $\dimv T(i)$.
\end{proof}

Dually, taking the universal $\add(T)$-extension of $D\Lambda$
\[ \varepsilon_{D\Lambda}\colon 0 \to {}_{D\Lambda}T' \to {}_{D\Lambda}E \to D\Lambda \to 0, \]
we see that $T\oplus {}_{D\Lambda}E$ is rigid. Since each injective indecomposable has an $\add(T)$-resolution of length $1$, we see that the dimension vectors of the indecomposable summands of $T\oplus {}_{D\Lambda}E$ span the Grothendieck group. Hence $T\oplus {}_{D\Lambda}E$ is tilting by Proposition \ref{r=n}. We can therefore define the dual Bongartz completion $T\oplus C$ and the dual Bongartz complement $C$.

\begin{Lem}\label{Bongartz-generates}
Suppose $\Ext^1(T,M)=0$. Then $\Ext^1(B,M)=0$ and $T\oplus B$ generates $M$.
\end{Lem}

\begin{proof}
Let $0\to\Lambda\to E_\Lambda\to T'_\Lambda\to 0$ be a universal $\add(T)$-coextension of $\Lambda$. Applying $\Hom(-,M)$ shows that $\Ext^1(E_\Lambda,M)=0$. Also, we know that $\Lambda^r\twoheadrightarrow M$ for some $r$. Since $\Ext^1(T,M)=0$, this map factors through $\Lambda^r\to E_\Lambda^r$, so $E_\Lambda$ generates $M$.
\end{proof}

The next result is a slight generalisation of key results in \cite{RS1,HU1}. If $\mathcal A$ is an additive subcategory of an abelian category $\mathcal C$, we define a category $\mathcal C/\mathcal A$ by quotienting out all morphisms factoring through $\mathcal A$.

\begin{Thm}\label{min-left-inj}
Let $M$ be coprime to $T$ and with $\Ext^1(T,M)=0$. Suppose that the minimal left $\add(T)$-approximation $\lambda_M\colon M\to T'$ is a monomorphism. Then
\begin{enumerate}
\item the cokernel $\rho_N\colon T'\to N$ is a minimal right $\add(T)$-approximation of $N$;
\item $M$ is a summand of the Bongartz complement of $T$, and $N$ is a summand of the dual Bongartz complement of $T$;
\item $\End(M)/\add(T)\cong\Ext^1(N,M)$ as left $\End(M)$-modules, and dually for $N$. Thus $\End(M)/\add(T)\cong\End(N)/\add(T)$ as algebras. In particular, if $M$ is indecomposable, then so is $N$, and $\End(M)\cong\End(N)$.
\end{enumerate}
\end{Thm}

\begin{proof}
Consider the short exact sequence
\[ \varepsilon\colon 0\to M\xrightarrow{\lambda_M} T'\xrightarrow{\rho_N} N\to 0. \]
Applying $\Hom(T,-)$ gives that $\Ext^1(T,N)=0$ and that $\rho_{N\,\ast}\colon\Hom(T,T')\to\Hom(T,N)$ is onto, hence is a right $\add(T)$-approximation. Moreover, since $\lambda_M$ is minimal, $\rho_N$ is radical, so $N$ and $T$ are coprime. Dually, since $M$ and $T$ are coprime, $\lambda_M$ is radical, so $\rho_N$ is minimal.

Applying $\Hom(-,T)$ and using that $\lambda_M$ is a left $\add(T)$-approximation, we see that $\Ext^1(M\oplus N,T)=0$. 

Applying $\Hom(-,M)$ shows that $M$ is rigid, whence $T\oplus M$ is rigid. Let $B$ be the Bongartz complement to $T$. We know that $\Ext^1(B,M)=0$ by Lemma \ref{Bongartz-generates}, whereas applying $\Hom(-,B)$ gives $\Ext^1(M,B)=0$. Thus $M\oplus T\oplus B$ is rigid, so $M$ and $T$ coprime implies $M\in\add(B)$.

Furthermore, we have an exact sequence of left $\End(M)$-modules
\[ \Hom(T',M) \to \End(M) \to \Ext^1(N,M) \to 0. \]
Since $\lambda_M$ is a left $\add(T)$-approximation, all endomorphisms of $M$ factoring through $\add(T)$ factor through $\lambda_M$. Thus $\Ext^1(N,M)\cong\End(M)/\add(T)$.

The statements for $N$ are dual.

To obtain the isomorphism $\End(M)/\add(T)\cong\End(N)/\add(T)$, we send an endomorphism $\theta$ of $M$ to the unique (modulo $\add(T)$) endomorphism $\phi$ of $N$ such that $\theta\varepsilon=\varepsilon\phi$. In particular, $\End(M)/\mathrm{rad}\End(M)\cong\End(N)/\mathrm{rad}\End(N)$. So, if $M$ is indecomposable, then $\End(M)$ is a skew-field by Corollary \ref{mono-or-epi}, whence $N$ is indecomposable and $\End(M)\cong\End(N)$.
\end{proof}

\begin{Cor}\label{min-right-Bongartz}
Let $B$ be the Bongartz complement of $T$. Then the minimal right $\add(T)$-approximation $\rho_B$ of $B$ is injective.
\end{Cor}

\begin{proof}
Let $N$ be an indecomposable summand of $B$. If $\rho_N$ were surjective, then by the dual of Theorem \ref{min-left-inj}, its kernel $M$ would also be an indecomposable summand of $B$, so $\Ext^1(N,M)=0$, a contradiction.
\end{proof}

We can decompose the Bongartz complement of $T$ as $B=B_1\oplus B_2$ such that the minimal left $\add(T)$-approximation $\lambda_{B_1}\colon B_1\to {}_{B_1}T$ of $B_1$ is a monomorphism, and the minimal left $\add(T)$-approximation $\lambda_{B_2}\colon B_2\to {}_{B_2}T$ of $B_2$ is an epimorphism.

Recall that $\sigma_T=\{i:\Hom(P(i),T)=0\}$ and $P_T:=\bigoplus_{i\in\sigma_T}P(i)$.

\begin{Prop}\label{B_1,B_2 Prop}
\begin{enumerate}
\item $B_1$ is the relative Bongartz complement of $T$; that is, $B_1$ is the Bongartz complement of $T$ inside $\modcat\Lambda_T$.
\item We have a short exact sequence
\[ 0 \longrightarrow P_T \longrightarrow B_2\xrightarrow{\lambda_{B_2}} {}_{B_2}T \longrightarrow 0. \]
Moreover, $\Hom({}_{B_2}T,B_2)=0$ and $\End(B_2)\cong\End(P_T)$.
\end{enumerate}
\end{Prop}

\begin{proof}
From the construction of the Bongartz complement, we have the universal $\add(T)$-coextension
\[ \eta_\Lambda\colon 0 \to \Lambda \to E_\Lambda \to T'_\Lambda \to 0. \]
Consider the short exact sequence
\[ 0 \longrightarrow K \longrightarrow B_2 \xrightarrow{\lambda_{B_2}} {}_{B_2}T \longrightarrow 0. \]
Since $B_2$ is a summand of $E_\Lambda$, we have the composition $B_2\to E_\Lambda\to T'_\Lambda$. This factors through the left $\add(T)$-approximation $\lambda_{B_2}$, and hence induces a monomorphism $K\to\Lambda$, so $K$ is projective. Applying $\Hom(-,T)$ shows that $\Hom(K,T)=0$, so $K\in\add(P_T)$.

We next observe that $\Hom({}_{B_2}T,B_2)=0$. For, if $X$ is an indecomposable summand of $B_2$, then $\lambda_X$ is an epimorphism by assumption, and $\rho_X$ is a monomorphism by Corollary \ref{min-right-Bongartz}. Let $Y$ be an indecomposable summand of ${}_{B_2}T$. Then there exists a non-zero map $X\to Y$ for some $X$, and this is onto since it factors through $\lambda_X$. If we have a non-zero map $Y\to X'$, then this will be injective since it factors through $\rho_{X'}$, but then $X\twoheadrightarrow Y\hookrightarrow X'$ will be neither a monomorphism nor an epimorphism, contradicting Corollary \ref{mono-or-epi}.

Applying $\Hom(-,B_2)$ now yields that $\Hom(K,B_2)\cong\End(B_2)$, and applying $\Hom(K,-)$ shows that $\Hom(K,B_2)\cong\End(K)$. Hence $\End(K)\cong\End(B_2)$. Since $B_2$ is basic, so too is $K$. Therefore $K$ is a direct summand of $P_T$ and $B_2$ has rank at most $|\sigma_T|$.

On the other hand, $T\oplus B_1$ is a basic rigid $\Lambda_T$-module, so has rank at most $n-|\sigma_T|$. Since $T\oplus B_1\oplus B_2$ is tilting, it has rank $n$, so we must have equalities above. It follows that $K\cong P_T$ and that $T\oplus B_1$ is a basic tilting $\Lambda_T$-module.

Let $B'$ be the Bongartz complement relative to $\Lambda_T$. Since $T\oplus B_1$ is a rigid $\Lambda_T$-module, we have $\Ext^1(B',B_1)=0$ by Lemma \ref{Bongartz-generates}. Similarly, since $T\oplus B'$ is a rigid $\Lambda$-module, we have $\Ext^1(B_1,B')=0$. Thus $T\oplus B_1\oplus B'$ is a rigid $\Lambda_T$-module, whence $B_1\cong B'$.
\end{proof}

Let $C$ be the dual Bongartz complement of $T$. Then we can write $C=C_1\oplus C_2$ such that $C_1$ is the relative dual Bongartz complement.

\begin{Cor}\label{B-C-bijection}
There are short exact sequences
\[ 0\to P(i)\to B(i)\xrightarrow{\lambda_{B(i)}} {}_{B(i)}T\to 0 \quad\textrm{and}\quad
0\to T_{C(i)}\xrightarrow{\rho_{C(i)}} C(i)\to I(i)\to 0 \]
yielding bijections between the indecomposable summands of $B_2$, the indecomposable summands of $C_2$, and the vertices in $\sigma_T$.

Similarly, there are short exact sequences
\[ 0\longrightarrow B(j)\xrightarrow{\lambda_{B(j)}} T_j\xrightarrow{\rho_{C(j)}} C(j)\longrightarrow 0\]
yielding bijections between the indecomposable summands of $B_1$ and those of $C_1$.
\end{Cor}

\begin{proof}
The first statement follows from Proposition \ref{B_1,B_2 Prop} and its dual, the second by Theorem \ref{min-left-inj}.
\end{proof}

\section{Perpendicular Categories}

Again let $T$ be rigid. Following Geigle and Lenzing \cite{GL}, we define the left perpendicular category to be
\[ {}^\perp T := \{M:\Hom(M,T)=0=\Ext^1(M,T)\}. \]
It is clear that this is an exact abelian subcategory which is closed under extensions, and hence hereditary.

The next theorem generalises a result of Ringel \cite{Ringel1}, who proved it for exceptional modules.

\begin{Thm}\label{perp-thm}
There is an equivalence of categories
\[ \{M:\Ext^1(T,M)=0, \lambda_M\textrm{ epi}\}/\add(T) \overset{F}{\underset{G}{\rightleftarrows}} \{X:\Hom(X,T)=0\} \]
such that $F(M):=\Ker(\lambda_M)$ and
\[ \eta_X\colon0\longrightarrow X\xrightarrow{c_X} G(X) \xrightarrow{d_X} T'_X \longrightarrow 0 \]
is the (minimal) universal $\add(T)$-coextension of $X$.

This induces an equivalence of categories
\[ \{M:\Ext^1(T,M)=0=\Ext^1(M,T), \lambda_M\textrm{ epi}\}/\add(T) \cong {}^\perp T. \]
\end{Thm}

\begin{proof}
Consider first the functor $F$. Since $\lambda_M$ is a left $\add(T)$-approximation and $T$ is rigid, we see that $\Hom(F(M),T)=0$. Applying $\Hom(F(M),-)$ to $\lambda_N$ and $\Hom(-,N)$ to $\lambda_M$ shows that
\[ \Hom(F(M),F(N)) \cong \Hom(F(M),N) \cong \Hom(M,N)/\add(T). \]
Thus $F$ defines a fully faithful functor.

Next consider the functor $G$. Since $\Hom(X,T)=0$ we see that $d_X$ is a left $\add(T)$-approximation of $G(X)$ (and minimal since it is surjective). Also, since $\eta_X$ is universal and $T$ is rigid, we must have $\Ext^1(T,G(X))=0$. Applying $\Hom(-,G(Y))$ to $\eta_X$, and $\Hom(X,-)$ to $\eta_Y$, shows that
\[ \Hom(G(X),G(Y))/\add(T) \cong \Hom(X,G(Y)) \cong \Hom(X,Y). \]
Thus $G$ also defines a fully faithful functor.

Since $d_X$ is a minimal left $\add(T)$-approximation of $G(X)$, we have $FG(X)\cong X$. Similarly, since $\Ext^1(T,M)=0$, the sequence
\[ 0 \longrightarrow F(M) \longrightarrow M \overset{\lambda_M}{\longrightarrow} {}_MT \longrightarrow 0 \]
is a universal $\add(T)$-coextension of $F(M)$. Thus $GF(M)\cong M$ modulo $\add(T)$ by Corollary \ref{ext-corollary}. This proves that $F$ and $G$ are inverse equivalences.

The final statement follows by noting that $\Ext^1(M,T)\cong\Ext^1(F(M),T)$ and $\Ext^1(X,T)\cong\Ext^1(G(X),T)$.
\end{proof}

Recall that $B=B_1\oplus B_2$, where $B_1$ is the relative Bongartz complement of $T$. Also, by Corollary \ref{min-right-Bongartz}, the minimal right $\add(T)$-approximation of $B$ is a monomorphism. We let $\overline B=\overline B_1\oplus\overline B_2$ denote the cokernel of this minimal right $\add(T)$-approximation, so we have a short exact sequence
\[ 0\longrightarrow T_B\xrightarrow{\rho_B} B\longrightarrow\overline B\longrightarrow 0. \]
We make the dual definitions for $C$.

\begin{Thm}\label{proj-gen-inj-cogen}
The right perpendicular category $T^\perp$ has projective generator $\overline B$ and injective cogenerator $\tau\overline C_1\oplus I_T$. In particular, $\End(\overline B)\cong\End(\tau\overline C_1\oplus I_T)$ is a basic hereditary artin $R$-algebra.
\end{Thm}

\begin{proof}
By the dual of Theorem \ref{perp-thm}, $\overline B\in T^\perp$. Also, by Lemma \ref{Bongartz-generates}, we know that every module $M\in T^\perp$ satisfies $\Ext^1(B,M)=0$ and is generated by $T\oplus B$, hence by $B$. Applying $\Hom(-,M)$ to the short exact sequence for $\overline B$ shows that $M$ is generated by $\overline B$ and that $\Ext^1(\overline B,M)=0$. Thus $\overline B$ is a projective generator for $T^\perp$.

Dually, $\overline C$ is an injective cogenerator for ${}^\perp T$. Now, the Auslander-Reiten formulae
\[ \Hom(X,T)\cong D\Ext^1(T,\tau X) \quad\textrm{and}\quad \Ext^1(X,T)\cong D\Hom(T,\tau X) \]
show that there is an equivalence of categories
\[ \tau\colon {}^\perp T/\add(P_T) \xrightarrow{\sim} T^\perp/\add(I_T). \]
In particular, $\Ext^1(M,\tau\overline C_1)=0$ for all $M\in T^\perp$ coprime to $I_T$. Also, since
\[ \Ext^1(I_T,\tau\overline C_1)\cong D\Hom(\overline C_1,I_T)\cong D\Hom(C_1,I_T)=0, \]
and
\[ \Hom(P_T,\overline C_1)\cong\Hom(P_T,C_1)=0 \]
we deduce that $\tau\overline C_1$ is a relative injective in $T^\perp$ and that $\tau^-\tau\overline C_1\cong\overline C_1$. Clearly $I_T\in T^\perp$ is also a relative injective, and $\tau\overline C_1\oplus I_T$ is basic of rank $n-\rank T$. Therefore $\tau\overline C_1\oplus I_T$ is an injective cogenerator for $T^\perp$.

We therefore have
\[ \End(\overline B)\cong\End(\tau\overline C_1\oplus I_T) \quad\textrm{and}\quad T^\perp \cong\modcat\End(\overline B). \]
Since $T^\perp$ is an hereditary $R$-category, we deduce that $\End(\overline B)$ is a basic hereditary artin $R$-algebra.
\end{proof}

We now make a small digression on APR-tilts \cite{APR}. Write $\Lambda_\Lambda=X\oplus Y$, where $X$ is simple and not injective. Then
\[ \Lambda \cong \begin{pmatrix}\End(X) & 0\\\Hom(X,Y)&\End(Y)\end{pmatrix}. \]
The Auslander-Platzeck-Reiten tilt of $\Lambda$ at $X$ is the algebra $\Gamma:=\End(\tau^-X\oplus Y)$. We observe that
\[ \Hom(\tau^-X,Y)\cong D\Ext^1(Y,X)=0 \quad\textrm{and}\quad \Hom(Y,\tau^-X)\cong\Ext^1(\nu Y,X), \]
where $\nu:=D\Hom(-,\Lambda)$ is the Nakayama functor (so that $\nu Y$ is the injective cover of the semisimple module $Y/\mathrm{rad}Y$). This latter isomorphism can be seen as follows. We know that $Y$ is a projective right $\Lambda$-module, so that $\Hom(Y,\Lambda)$ is a projective left $\Lambda$-module. Therefore
\begin{align*}
\Hom(Y,\tau^-X) &= \Hom(Y,\Ext^1(DX,\Lambda)) \cong \Ext^1(DX,\Lambda)\otimes_\Lambda\Hom(Y,\Lambda)\\
&\cong \Ext^1(DX,\Hom(Y,\Lambda)) \cong \Ext^1(D\Hom(Y,\Lambda),X) = \Ext^1(\nu Y,X).
\end{align*}
N.B. This is used to show that $\tau Y\cong\nu Y[-1]$ in the derived category of $\modcat\Lambda$ \cite{Happel}.

Since $\End(Y)\cong\End(\nu Y)$ and $\End(X)\cong\End(\tau^-X)$, we have that
\[ \Gamma \cong \begin{pmatrix}\End(X)&\Ext^1(\nu Y,X)\\0&\End(\nu Y)\end{pmatrix}. \]

\begin{Lem}
Suppose we can write $\Lambda_\Lambda=X\oplus Y$ with $\Hom(Y,X)=0$, so that
\[ \Lambda \cong \begin{pmatrix}\End(X) & 0\\\Hom(X,Y)&\End(Y)\end{pmatrix}. \]
Then the algebra
\[ \Gamma := \begin{pmatrix}\End(X)&\Ext^1(\nu Y,X)\\0&\End(\nu Y)\end{pmatrix} \]
can be obtained from $\Lambda$ by a sequence of APR-tilts.
\end{Lem}

\begin{proof}
Since $\add(X)$ is closed under submodules, we can write $X=X_1\oplus X_2$, where $X_1$ contains all submodules of all projective-injective summands of $X$. Then $\Lambda\cong\End(X_1)\times\Lambda'$ and $\Gamma\cong\End(X_1)\times\Gamma'$, where
\[ \Lambda' \cong \begin{pmatrix}\End(X_2) & 0\\\Hom(X_2,Y)&\End(Y)\end{pmatrix} \quad\textrm{and}\quad \Gamma' \cong \begin{pmatrix}\End(X_2)&\Ext^1(\nu Y,X_2)\\0&\End(\nu Y)\end{pmatrix} \]
We now observe that, as in the preceding comments, $\Gamma'\cong\End(\tau^-X_2\oplus Y)$. Therefore $\Gamma\cong\End(T)$, where $T:=X_1\oplus\tau^-X_2\oplus Y$. We note that $\Ext^1(T,T)=0$, so that $T$ is a tilting module, and that $\Ext^1(\tau T,T)=0$, so that $T$ is a slice in the Auslander-Reiten quiver of $\Lambda$. Thus $\Gamma$ is hereditary and can be obtained from $\Lambda$ by a sequence of APR-tilts \cite{Ringel3}.
\end{proof}

\begin{Thm}\label{derived-equiv}
The hereditary categories $T^\perp$ and ${}^\perp T$ are derived equivalent. In fact, $\End(\overline B)$ and $\End(\overline C)$ are related by a sequence of APR-tilts.
\end{Thm}

\begin{proof}
We know that $\overline B=\overline B_1\oplus\overline B_2$ is a projective generator for $T^\perp$. We set
\[ \Gamma:=\End(\overline B_1\oplus\overline B_2). \]
Similarly, $\tau^-\overline B_1\oplus P_T$ is a projective generator for ${}^\perp T$. We set
\[ \Gamma':=\End(\tau^-\overline B_1\oplus P_T)\cong\End(\overline C). \]

Now, using the short exact sequences
\[ 0\to P_T\to B_2 \to {}_{B_2}T\to 0 \quad\textrm{and}\quad 0 \to T_{B_2}\to B_2\to \overline B_2\to 0, \]
we obtain that
\[ \Hom(\overline B_2,\overline B_1)\cong\Hom(B_2,\overline B_1)\cong\Hom(P_T,\overline B_1)=0. \]
Also, it is clear that $\nu P_T\cong I_T$ and that $\Hom(\tau^-\overline B_1,P_T)=0$. Thus
\[ \Gamma'\cong\begin{pmatrix}\End(\overline B_1)&\Ext^1(I_T,\overline B_1)\\0&\End(I_T)\end{pmatrix}. \]
By the previous lemma, it only remains to show that $\bar\nu(\overline B_2)\cong I_T$, where $\bar\nu$ is the Nakayama functor for $T^\perp$.

We know that $\bar\nu\overline B=\tau\overline C_1\oplus I_T$. Applying $\Hom(-,I_T)$ to the short exact sequence for $\overline B_1$ gives that
\[ \Hom(\overline B_1,I_T)\cong\Hom(B_1,I_T)=0. \]
Thus, inside $T^\perp$, $\bar\nu \overline B_1 \cong\overline C_1$ and $\bar\nu \overline B_2 \cong I_T$ as required.
\end{proof}

\section{The Cluster Complex as a Simplicial Polytope}

We begin by recalling the definition of an abstract simplicial polytope \cite{MS}.

Let $\mathcal P$ be a poset, whose elements are called faces, such that
\begin{enumerate}
\item[AP1] there are unique maximal and minimal faces.
\item[AP2] each flag (maximal chain) has length $n+1$. We set $\rk\mathcal P:=n$.
\end{enumerate}
Given faces $F\leq G$, we define the section
\[ G/F := \{H:F\leq H\leq G\}. \]
Thus each section is again a poset satisfying AP1 and AP2.

The poset $\mathcal P$ is called an abstract polytope if two further axioms are satisfied:
\begin{enumerate}
\item[AP3] $\mathcal P$ is strongly flag connected.
\item[AP4] if $\rk G/F=1$, then the poset contains precisely four elements.
\end{enumerate}
We observe that AP4 is inherited by sections. Also, if AP4 is satisfied and $F_{-1}<F_0<\cdots<F_i<\cdots<F_n$ is a flag, then there exists a unique face $\widehat F_i$ such that $F_{-1}<F_0<\cdots<\widehat F_i<\cdots<F_n$ is again a flag. We say that two such flags are adjacent. Now, adjacency induces an equivalence relation on the set $\mathcal F(\mathcal P)$ of all flags, and we say that $\mathcal P$ is flag connected if any two flags are equivalent, so one can be transformed into the other via a sequence of adjacent flags. Finally, we say that $\mathcal P$ is strongly flag connected if each section of $\mathcal P$ is flag connected. It follows that each section is again an abstract polytope.

We may identify a face $F$ with the section $F/F_{-1}$, and so define $\rk F:=\rk F/F_{-1}$. Note that $\rk G/F=\rk G-\rk F-1$. The section $F_n/F$ is called the co-face at $F$.

We call $\mathcal P$ simplicial if each facet (rank $n-1$) is adjacent to $n$ ridges (rank $n-2$). Equivalently, each proper face $F$ is isomorphic to an $r$-simplex, where $r=\rk F$. In this case, $\mathcal P\setminus F_n$ is a simplicial complex. We observe that each regular polygon is a simplicial polytope, but only the triangle is a simplicial complex.

\medskip

We now define a poset $\mathcal T=\mathcal T_\Lambda$ using the rigid modules for $\Lambda$. Let
\[ \mathcal E:=\{1,\ldots,n\}\cup\{\textrm{exceptional modules}\}/\cong \]
and define a poset $\mathcal T'\subset\mathcal P(\mathcal E)$ such that $\{T(1),\ldots,T(r),j_1,\ldots,j_s\}\in\mathcal T'$ provided $T:=\bigoplus_iT(i)$ is rigid and $\sigma:=\{j_1,\ldots,j_s\}\subset\sigma_T$. We write $\mathcal T:=\mathcal T'\cup\{\mathcal E\}$ and set
\[ \rk(T,\sigma):=\rank T+|\sigma|-1, \quad \rk\mathcal E:=n. \]
We note that there is a discrepancy between the rank of a tilting module and its rank as a face of the polytope. Hence we shall use $\rk$ when we mean the rank as a face of the polytope.

\begin{Thm}\label{Main Thm}
Let $\Lambda$ be a basic hereditary artin algebra. Then the poset $\mathcal T$ is an abstract simplicial polytope, called the cluster complex of $\Lambda$.
\end{Thm}

The rest of this section will be devoted to the proof of this theorem.

We observe that the empty set is the unique minimal face (having rank $-1$) and that $\mathcal E$ is the unique maximal face (having rank $n$), so AP1 is satisfied.

Let $(T,\sigma)\in\mathcal T'$ be maximal. Then clearly $\sigma=\sigma_T$ and $T$ must be a tilting $\Lambda_T$-module (for otherwise we could take the relative Bongartz completion of $T$). Thus, by Proposition \ref{r=n}, $\rk(T,\sigma)=n-1$.

We call a rigid module $T$ support-tilting if it is a tilting module for $\Lambda_T$. It follows that the facets of $\mathcal T$ are in bijection with the isomorphism classes of basic support-tilting modules. In particular, the trivial support-tilting module $0$ corresponds to the facet $\{1,\ldots,n\}$.

\begin{Lem}\label{simplex}
Given faces $F<G$ in $\mathcal T'$, the section $G/F$ is isomorphic to a $k$-simplex, where $k=\rk G/F$. Thus $\mathcal T'$ is a pure simplicial complex of rank $n-1$.
\end{Lem}

\begin{proof}
Let $F=(T,\sigma)$ and $G=(T\oplus T',\sigma\cup\sigma')$. Then the faces of $G/F$ are those of the form $(T\oplus U,\sigma\cup\rho)$ such that $U$ is a direct summand of $T'$ and $\rho\subseteq\sigma'$. Hence the section $G/F$ is isomorphic to the power set of $\{0,1,2,\ldots,k\}$, where $k=\rk(T',\sigma')=\rk G/F$.
\end{proof}

We extend the definition of perpendicular categories to faces of $\mathcal T$ by setting
\[ (T,\sigma)^\perp := T^\perp\cap\modcat\Lambda_\sigma \quad\textrm{and}\quad {}^\perp(T,\sigma) := {}^\perp T\cap\modcat\Lambda_\sigma. \]
By first restricting to $\modcat\Lambda_\sigma$, Theorem \ref{perp-thm} implies that
\[ {}^\perp(T,\sigma) \cong \{M\in\modcat\Lambda_\sigma:\Ext^1(T,M)=0=\Ext^1(M,T), \lambda_M\textrm{ epi}\}/\add(T). \]
Moreover, by Theorem \ref{proj-gen-inj-cogen} and its dual, both ${}^\perp(T,\sigma)$ and $(T,\sigma)^\perp$ are equivalent to module categories of basic hereditary artin algebras having rank $n-\rk(T,\sigma)-1=\rk\mathcal E/(T,\sigma)$.

\begin{Thm}\label{co-face}
Let $(T,\sigma)$ be a face of $\mathcal T$. Then the co-face $\mathcal E/(T,\sigma)$ is isomorphic to $\mathcal T_{{}^\perp(T,\sigma)}$, and also to $\mathcal T_{(T,\sigma)^\perp}$.
\end{Thm}

\begin{proof}
By restricting first to $\modcat\Lambda_\sigma$, we may assume that $\sigma=\emptyset$. It is clear that the basic rigid modules in $\mathcal E/T$ are in bijection with the basic rigid modules in
\[ \{M:\Ext^1(T,M)=0=\Ext^1(M,T)\}/\add(T). \]
In the same way that we constructed the poset $\mathcal T$ from $\mathcal E$, we may construct a poset $\widetilde{\mathcal T}$ from
\[ \widetilde{\mathcal E} := \{i\in\sigma_T\}\cup\{M\textrm{ exceptional}, T\oplus M\textrm{ basic rigid}\}/\cong. \]
Then the section $\mathcal E/T$ is isomorphic to $\widetilde{\mathcal T}$.

We want to show that the equivalence of categories
\[ \{M:\Ext^1(T,M)=0=\Ext^1(M,T),\lambda_M\textrm{ epi}\}/\add(T) \cong {}^\perp T \]
induces an isomorphism of posets $\widetilde{\mathcal T} \cong \mathcal T_{{}^\perp T}$.

We recall from Theorem \ref{min-left-inj} and Proposition \ref{B_1,B_2 Prop} that the exceptional modules $M$ such that $T\oplus M$ is basic rigid and $\lambda_M$ is a monomorphism are precisely the summands of $B_1$, the relative Bongartz completion of $T$.

Given an exceptional module $X$ in ${}^\perp T$, we have an exceptional module $G(X)$ such that $T\oplus G(X)$ is basic rigid, using the functor $G$ from Theorem \ref{perp-thm}. It remains to define $G$ on the vertices of ${}^\perp T$. We recall that ${}^\perp T$ has projective generator $\tau^-\overline B_1\oplus P_T$ and injective cogenerator $\overline C$. Thus the vertices of ${}^\perp T$ are in bijection with the summands $B(j)$ of $B_1$ together with the vertices in $\sigma_T$. We observe that $G(\overline C)=C$ and $G(P_T)=B_2$ by construction. If $i\in\sigma_T$, then we define $G(i):=i$. Otherwise, if $B(j)$ is a summand of $B_1$, then we define $G(j):=B(j)$.

Conversely, let $M$ be an exceptional module such that $T\oplus M$ is basic rigid. If $M$ is not a summand of $B_1$, then $\lambda_M$ is an epimorphism, so we have an exceptional module $F(M)\in{}^\perp T$. If $M=B(j)$ is a summand of $B_1$, we define $F(B(j)):=j$, the corresponding vertex of ${}^\perp T$. Finally, if $i\in\sigma_T$, we define $F(i):=i$, the vertex in ${}^\perp T$.

It is now clear that $F$ and $G$ are inverse bijections between the vertices of $\widetilde{\mathcal T}$ and of $\mathcal T_{{}^\perp T}$. It remains to show that they induce bijections on the whole posets.

Let $(X,\sigma)$ be a face of $\mathcal T_{{}^\perp T}$. We can write $\sigma=\sigma'\cup\sigma''$, where $\sigma''=\sigma\cap\sigma_T$. Set $B':=G(\sigma')$, a summand of $B_1$, and $P'':=\bigoplus_{i\in\sigma''}P(i)$, a summand of $P_T$. Then, since $(X,\sigma)$ is a face of $\mathcal T_{{}^\perp T}$,
\[ \Hom(P'',X)=0 \quad\textrm{and}\quad \Hom(\tau^-\overline B',X)=0. \]
Using the universal coextension $\eta_X$, we deduce that
\[ \Hom(P'',G(X))=0 \quad\textrm{and}\quad \Ext^1(G(X),\overline B')\cong D\Hom(\tau^-\overline B',G(X))=0. \]
Finally, using $\rho_{B'}$, we see that $\Ext^1(G(X),B')=0$.

We know that, since $T\oplus G(X)$ is rigid, $\Ext^1(B',G(X))=0$. We have therefore shown that $G(X,\sigma)=(G(X)\oplus B',\sigma'')$ is a face of $\widetilde{\mathcal T}$.

The proof that $F$ sends faces to faces is entirely analogous. Since $F$ and $G$ clearly preserve the partial order, we are done.
\end{proof}

By Lemma \ref{simplex}, we know that $\mathcal T'$ is a pure simplicial complex of rank $n-1$, and hence that every flag in $\mathcal T'$ has length $n$. Therefore every flag in $\mathcal T$ has length $n+1$, so that AP2 holds. Moreover, AP3 and AP4 hold for all sections in $\mathcal T'$, since all such sections are simplices. On the other hand, using Theorem \ref{co-face} and induction on rank, we see that AP3 and AP4 hold for all proper sections. We are left with proving that the poset $\mathcal T$ is itself flag connected, as well as that AP3 and AP4 hold whenever $\rank\Lambda=1$.

We recall the classification of rigid modules over hereditary artin algebras of rank 1 or 2.

Suppose first that $\Lambda$ has rank 1, so that $\Lambda_\Lambda$ is a simple module. Thus $\Lambda$ is a division algebra, whence $\mathcal E=\{\Lambda,1\}$ and $\mathcal T_\Lambda=\{\emptyset,1,\Lambda,\mathcal E\}$. Therefore $\mathcal T_\Lambda$ is a polytope, in fact a 1-simplex (line segment). In particular, AP3 and AP4 are satisfied.

Now suppose that $\Lambda$ has rank 2. Then $\Lambda_\Lambda\cong P(1)\oplus P(2)$ with $P(2)$ simple and $\Hom(P(1),P(2))=0$. Thus
\[ \Lambda \cong \End(\Lambda) \cong \begin{pmatrix} \End(P(1)) & \Hom(P(2),P(1))\\0 & \End(P(2))\end{pmatrix} = \begin{pmatrix} U&M\\0&V\end{pmatrix}, \]
where $U$ and $V$ are division algebras and $M$ is an $U$-$V$-bimodule. Set
\[ u:=\ell_R(U), \quad v:=\ell_R(V), \quad m:=\ell_R(M), \quad r:=\dim_UM, \quad s:=\dim_VM. \]
Then $ru=m=sv$, and the Euler form on $K_0$ is given by the matrix $\left(\begin{smallmatrix}u&-m\\0&v\end{smallmatrix}\right)$ with respect to the standard basis. In fact, since $\mathrm{rad}P(1)\cong M$, we have that $\Ext^1(S(1),S(2))\cong\Hom_V(M,V)$ as $V$-$U$-bimodules.

The symmetrisation $\left(\begin{smallmatrix}2u&-m\\-m&2v\end{smallmatrix}\right)$ is a symmetrisable generalised Cartan matrix
\[ \begin{pmatrix}2&-r\\-s&2\end{pmatrix} = \begin{pmatrix}u&0\\0&v\end{pmatrix}^{-1}\begin{pmatrix}2u&-m\\-m&2v\end{pmatrix}, \]
so has an associated root system. The APR-tilting functors, when composed with the standard duality, give endofunctors of $\modcat\Lambda$ whose action on dimension vectors corresponds to the simple reflections on the root system. In this way, we can apply Kac's arguments \cite{Kac1,Kac2} to show that the dimension vector of an indecomposable $\Lambda$-module is necessarily a root and that the set of dimension vectors of exceptional modules is precisely the set of positive real roots. (The usual difficulties about admissibility of reflections does not arise, since each vertex is either a sink or a source.)

We deduce that the only exceptional modules are the preprojective modules $\tau^{-t}P(i)$ and the preinjective modules $\tau^tI(i)$, for $i=1,2$ and $t\geq0$. Finally, we see that the tilting modules are given by pairs of adjacent modules in the preprojective or preinjective components of the Auslander-Reiten quiver. More explicitly, they are given by
\[ \tau^{-t}(P(2)\oplus P(1)), \quad \tau^{-t}(P(1)\oplus\tau^-P(2)), \quad \tau^t(I(2)\oplus I(1)), \quad \tau^t(\tau I(1)\oplus I(2)). \]
For $\tau^{-t}P(1)$ we note that its Bongartz completion is $\tau^{-t}P(2)$ and its dual Bongartz completion is $\tau^{-t-1}P(2)$. Similarly for the other exceptional modules.

We have the following dichotomy. If $rs\geq4$, then these indecomposables are all non-isomorphic, so that the poset $\mathcal T$ is isomorphic to $\mathbb Z$ with the usual linear order. In particular, $\mathcal T$ is an abstract simplicial polytope. On the other hand, if $rs\leq 3$, then $\mathcal T$ has only finitely many vertices. More precisely, $\mathcal T$ is isomorphic to an octagon if $rs=3$, a hexagon if $rs=2$, a pentagon if $rs=1$ and a square if $rs=0$. This proves that $\mathcal T$ is always a simplicial polytope when $\Lambda$ has rank 2.

We shall need the following observation concerning the preprojective and preinjective components when $\Lambda$ has rank 2 and $rs\geq4$.

\begin{Thm}\label{total-order}
Let $\Lambda$ be a representation-infinite basic hereditary artin algebra of rank 2. Let $d$ be a positive additive function on $\modcat\Lambda$, for example $d=\ell_R$. Then the ordering of the indecomposable preprojective modules given by
\[ P\leq P' \quad\textrm{if and only if}\quad \frac{d(P)}{\sqrt{\ell_R\End(P)}}\leq\frac{d(P')}{\sqrt{\ell_R\End(P')}} \]
coincides with the natural ordering coming from the Auslander-Reiten quiver. Dually for the preinjective component.
\end{Thm}

\begin{proof}
We keep the earlier notation, so that the Euler form of $\Lambda$ is given by the matrix $\left(\begin{smallmatrix}u&-m\\0&v\end{smallmatrix}\right)$, and $ru=m=sv$. Since $\Lambda$ is representation-infinite, we know that $rs\geq4$. We have Auslander-Reiten sequences of the form
\[ 0\to\tau^{1-t}P(1)\to(\tau^{-t}P(2))^s\to\tau^{-t}P(1)\to 0\]
and
\[ 0\to\tau^{-t}P(2)\to(\tau^{-t}P(1))^r\to\tau^{-t-1}P(2)\to 0. \]
Set $d(i,t):=d(\tau^{-t}P(i))$. It follows that
\[ d(1,t)=sd(2,t)-d(1,t-1) \quad\textrm{and}\quad d(2,t+1)=rd(1,t)-d(2,t). \]
Also, $\End(\tau^{-t}P(1))\cong U$ and $\End(\tau^{-t}P(2))\cong V$. Therefore, we wish to prove that
\[ d(2,t)/\sqrt v < d(1,t)/\sqrt u < d(2,t+1)/\sqrt v. \]
Replacing $u$ by $m/r$ and $v$ by $m/s$, it is enough to show that
\[ d(2,t)\sqrt s < d(1,t)\sqrt r < d(2,t+1)\sqrt s. \]

We observe that there is a short exact sequence
\[ 0\to P(2)^s\to P(1)\to S(1)\to 0. \]
Thus $d(1,0)=d(S(1))+sd(2,0)$, whence
\[ d(2,0)\sqrt s \leq sd(2,0) < d(1,0) \leq d(1,0)\sqrt r. \]

Suppose now that $d(2,t)\sqrt s < d(1,t)\sqrt r$. If we also have that $d(2,t+1)\sqrt s \leq d(1,t)\sqrt r$, then
\[ 2\big(d(2,t)+d(2,t+1)\big) = 2rd(1,t) > \big(d(2,t)+d(2,t+1)\big)\sqrt{rs} \geq 2\big(d(2,t)+d(2,t+1)\big), \]
a contradiction. Thus we must have that $d(1,t)\sqrt r < d(2,t+1)\sqrt s$ as required.

Analogously we can show that $d(2,t+1)\sqrt s < d(1,t+1)\sqrt r$, and so we are done by induction.
\end{proof}

\begin{Cor}\label{rank2}
Let $\Lambda$ be a connected basic hereditary artin algebra of rank 2 and let $d$ be a positive additive function on $\modcat\Lambda$. Let $T$ be a sincere exceptional module, with Bongartz complement $B$ and dual Bongartz complement $C$. Then either
\[ \frac{d(B)}{\sqrt{\ell_R\End(B)}}<\frac{d(T)}{\sqrt{\ell_R\End(T)}} \quad\textrm{or}\quad \frac{d(C)}{\sqrt{\ell_R\End(C)}}<\frac{d(T)}{\sqrt{\ell_R\End(T)}}. \]
\end{Cor}

\begin{proof}
Let $r$ and $s$ have their usual meanings. Then $\Lambda$ is connected if and only if $rs\geq1$. We have shown in the previous proposition that if $rs\geq4$, then the result holds. On the other hand, if $rs\leq 3$, then $\Lambda$ is representation-finite with at most six indecomposable modules, in which case a direct calculation proves the result.
\end{proof} 

Recall that we need to prove that the poset $\mathcal T$ is flag connected. Since $\mathcal T'$ is a simplicial complex, we know that all flags containing a given facet are equivalent.  We say that two support-tilting modules $T$ and $T'$ are mutations of one another provided that there is a flag containing $T$ which is adjacent to a flag containing $T'$. It is therefore enough to prove that each support-tilting module is mutation equivalent to the trivial support-tilting module $0$.

\begin{Prop}\label{Endos}
Let $T=\bigoplus_iT(i)$ and $T'=\bigoplus_iT'(i)$ be mutation-equivalent support-tilting modules. Then
\[ \prod_i\End(T(i))\times\prod_{i\in\sigma_T}\End(S(i))\cong\prod_i\End(T'(i))\times\prod_{i\in\sigma_{T'}}\End(S(i)). \]
\end{Prop}

\begin{proof}
It is enough to consider the case when $T'$ is a mutation of $T$. There are two cases to consider.

Suppose first that $\sigma_T=\sigma_{T'}$. By passing to $\Lambda_T$, we may assume that $T$ and $T'$ are sincere. Since we have flags which are adjacent, we can write $T=\overline T\oplus X$ and $T'=\overline T\oplus X'$ with $X$ and $X'$ exceptional. Using Theorem \ref{co-face}, we see that $\rk\mathcal E/\overline T=1$, so that $X$ and $X'$ are the only complements of $\overline T$. Therefore $\overline T$ is sincere and we may assume that $X$ is the Bongartz complement and $X'$ is the dual Bongartz complement. In particular, we can apply Theorem \ref{min-left-inj} to deduce that $\End(X)\cong\End(X')$. 

Instead, suppose that $T'=T\oplus X$ and $\sigma_T=\sigma_{T'}\cup\{i\}$. By passing to $\Lambda_{T'}$, we may assume that $T'$ is sincere and that $\sigma_T=\{i\}$. Again by Theorem \ref{co-face}, $X$ is the Bongartz complement of $T$ (and also the dual Bongartz complement) and by Proposition \ref{B_1,B_2 Prop} we see that $\End(X)\cong\End(P(i))\cong\End(S(i))$.
\end{proof}

For an exceptional module $M$ we define
\[ \mu(M):=\ell_RM/\sqrt{\ell_R\End(M)}, \]
and, for a support-tilting module $T=\bigoplus_{i=1}^rT(i)$, we set
\[ \lambda(T):=\big(0,\ldots,0,\mu(T(1)),\ldots,\mu(T(r))\big), \]
where we have $|\sigma_T|=n-r$ zeros and the $\mu(T(i))$ are assumed to be arranged in non-decreasing order. We observe that, by the proposition above, if $T'$ is equivalent to $T$, then the possible denominators which can occur in $\lambda(T')$ are fixed. Therefore, under the lexicographic ordering, the set of possible $\lambda(T')$ has no infinite descending chains.

If $T$ is insincere (i.e. $r<n$ above), then $T$ lies in the co-face $\mathcal E/(\sigma_T)$, which is isomorphic to $\mathcal T_{\Lambda_T}$. Since $\Lambda_T$ has smaller rank than $\Lambda$, the co-face is flag connected and so $T$ is mutation equivalent to $0$.

Otherwise, let $T$ be sincere and let $M=T(1)$ be an indecomposable summand of $T$ with $\mu(M)$ minimal. The co-face $\mathcal E/M$ is isomorphic to the poset $\mathcal T_{{}^\perp M}$ by Theorem \ref{co-face}. In particular, as was shown in the proof of that theorem, the zero module in ${}^\perp M$ corresponds to the relative Bongartz completion $B_1\oplus M$ of $M$. Thus, by induction on rank, $T$ is mutation equivalent to $B_1\oplus M$, and dually also to $M\oplus C_1$, where $C_1$ is the relative dual Bongartz complement.

Since $B_1$ is cogenerated by $M$, if $M$ is insincere, then $B_1\oplus M$ is insincere, whence $\lambda(B_1\oplus M)<\lambda(T)$. So, we may assume that $M$ is sincere (hence $B=B_1$ and $C=C_1$). As in Corollary \ref{B-C-bijection}, there exist short exact sequences
\[ 0 \to B(i) \to M^{r_i} \to C(i) \to 0, \]
where the $B(i)$ and $C(i)$ are indecomposable summands of $B$ and $C$ respectively. Now let $Z(i)$ be any tilting module in ${}^\perp(B(i)\oplus M)$. Then $Z(i)^\perp$ has rank 2 and $B(i)\oplus M$ and $M\oplus C(i)$ are tilting modules in $Z(i)^\perp$. Since the function $\ell_R$ on $\modcat\Lambda$ induces a positive additive function on $Z(i)^\perp$, we can apply Corollary \ref{rank2} to deduce that either $\mu(B(i))$ or $\mu(C(i))$ is smaller than $\mu(M)$. In this way, we have shown that $T$ is mutation equivalent to a support-tilting module $T'$ such that $\lambda(T')<\lambda(T)$. Since there are no such infinite decreasing sequences, we deduce that every support-tilting module is mutation equivalent to $0$.

This completes the proof that $\mathcal T$ satisfies AP3, and hence is an abstract simplicial polytope.

\begin{Cor}\label{Cor-Endos}
Let $T=\bigoplus_iT(i)$ be a basic support-tilting module. Then
\[ \prod_i\End(T(i))\cong\prod_{i\not\in\sigma_T}\End(S(i)). \]
\end{Cor}

\begin{proof}
This is immediate from Proposition \ref{Endos}, using that $\mathcal T$ is flag connected.
\end{proof}

{\small\noindent
Andrew Hubery\\
University of Leeds\\
ahubery@maths.leeds.ac.uk}

\end{document}